\documentclass[12pt]{amsart}

\usepackage{geometry,enumitem}
\usepackage{hyperref}

\usepackage{color}
\usepackage{accents}

\usepackage{amsthm, amsxtra}

\numberwithin{equation}{section}

\newtheorem{theorem}{Theorem}[section]
\newtheorem{theorem*}{Theorem}
\newtheorem{remark}[theorem]{Remark}
\newtheorem{lemma}[theorem]{Lemma}

\newtheorem{proposition}[theorem]{Proposition}
\newtheorem{corollary}[theorem]{Corollary}

\newtheorem*{question*}{Question}

\newtheorem{example}[theorem]{Example}

\newtheorem{definition}[theorem]{Definition}

\newcommand{\R}{\mathbb{R}}

\author{Young-Heon Kim, Brendan Pass, and David J.\ Schneider}

\address{Young-Heon Kim}
\address{Department of Mathematics\\ University of British Columbia\\ Vancouver, V6T 1Z2 Canada  }
\email{yhkim@math.ubc.ca}

\address{Brendan Pass}
\address{Department of Mathematical and Statistical Sciences\\632 CAB\\ University of Alberta\\
Edmonton, Alberta, Canada, T6G 2G1}
\email{pass@ualberta.ca}

\address{David J.\ Schneider}
\address{Global Institute for Food Security\\ University of Saskatchewan\\ 110 Gymnasium Place\\ Saskatoon, SK S7N 4J8}
\email{dave.schneider@gifs.ca}

\title[Optimal transport and barycenters for dendritic measures]{Optimal transport and barycenters for dendritic measures}

\thanks{The first two authors are partially supported by the Natural
  Sciences and Engineering Research Council of Canada (NSERC) through
  the Discovery Grant program. The third author is supported by the
  Canada Excellence in Research Chair in Food Systems and Security and
  the Global Institute for Food Security.\\
 The authors gratefully acknowledge the support of the Pacific Institute
of Mathematical Sciences (PIMS); in particular, they would like to thank PIMS director James
Colliander for initiating this project by bringing the three of them
together. Part of this work was done while
the authors participated in a research-in-teams program at the Banff
International Research Station (BIRS), and we thank BIRS for the excellent
working environment provided there.  \copyright 2019 by the authors.  }

\date{\today}

\begin{document}

\begin{abstract}
We introduce and study a variant of the Wasserstein distance on the
space of probability measures, specially designed to deal with measures
whose support has a dendritic, or treelike structure with a particular
direction of orientation.  Our motivation is the comparison of and
interpolation between plants' root systems.  We characterize
barycenters with respect to this metric, and establish that the
interpolations of root-like measures, using this new metric, are also
root like, in a certain sense; this property fails for conventional
Wasserstein barycenters.  We also establish geodesic convexity with
respect to this metric for a variety of functionals, some of which we
expect to have biological importance.
\end{abstract}

\maketitle

\section{Introduction}

In this paper we introduce a new metric on the space of probability
measures, the layerwise-Wasserstein distance. The motivation for this work is the need for a sound mathematical
framework for describing the structure and diversity of dendritic
structures in anisotropic environments.  In particular, we are
interested in the macroscopic structure of plant root systems
developing under the influence of gravity and the stratification of chemical constituents, texture and microbial activity characteristic of soils.  This biophysical context can be readily translated into mathematical terms.  Plant tissues are composed of cells that physically partition $\mathbb{R}^3$ into two connected components -- the ``inside'' and ``outside''.  The resulting structure roughly corresponds to a CW complex (see e.g., \cite{Hatcher2002}) describing the topology
of the plant.  Ignoring complex features present at microscopic
scales, the external surface can be viewed as a smooth, connected 2D manifold with genus zero embedded in $\mathbb{R}^3$.  Computational representations of these external surface can be reconstructed using
standard methods of optical, X-ray and neutron tomography.

This idealization misses two essential points: a) the above and below ground portions of plants display intricate structural forms that are remarkably resistant to quantitative analysis, and; b) the form and function of these complicated structures are intimately related to the anisotropic environment in which they develop.  The first condition implies the need to handle arbitrarily complicated distributions of mass in space subject to very modest restrictions on the behaviour of the surface while the second suggests the need to handle preferred
directions in space.

Natural challenges include quantifying the difference between two or more roots, summarizing or describing the typical structure of a family of root systems (for instance, the roots of several genetically identical plants, grown in nearly identical environments, which often exhibit considerable variation in their structure) in a succinct way, quantifying the variation within that family and comparing the structure exhibited by one family to another.  A typical approach to these problems is to compute a family of \textit{phenotypes} for each system (including, for example, total root length, rooting depth, and various topological invariants, such as the Horton-Strahler index) and compare and average among them (see, for instance,
\cite{bio1,bio2,bio3,bio4,bio5}).  Though this has met with some
success in distinguishing between particular choices of root systems, it is not generally clear which phenotypes are most useful for this purpose, and the choice in different applications is often done in an adhoc way. For virtually any collection of phenotypes, it is not hard to come up with drastically different root shapes sharing the same phenotypes.

Our approach focuses on roots as mass distributions in $\mathbb{R}^3$ where the vertical and horizontal directions have distinct roles (roots which are related by a rotation about the vertical axis are considered identical). Natural mathematical goals include constructing a metric between root shapes reflecting both their downward pointing dendritic topology as well as the distances and sizes in the underlying space\footnote{Ideally, the metric should detect geometric differences, between, for instance, a short limb and a long one, as well as topological differences, between say, a forked limb and a straight one.}, and producing a representative of a family of root systems which capture the average, or typical structure
among the family. After normalization for overall mass, root systems can be modeled as probability distributions; 
the Wasserstein distance from optimal transport \cite{Villani03,Villani09, Santambrogio15} is then one candidate for such a metric, and Wasserstein barycenters (Fr\'echet means with respect to this metric, see \cite{AguehCarlier11}) a corresponding candidate for a representative of a family.  While this metric has proved fruitful in related problems involving comparing and averaging among shapes (image processing, for instance), we demonstrate in this paper that it is not ideally adapted to the downward dendritic structure prominent among root systems, in large part because optimal matchings don't generally exhibit monotonicity in the distinguished, vertical direction.  While it is possible to incorporate vertical stratification 
 in the usual definition of Wasserstein distance by penalizing transport in the vertical direction, the practical application of this formalism is limited by computational requirements.  We propose a simple alternative based on a related metric, the layerwise -Wasserstein metric, derived from a variant of optimal transport  in which monotonicity in the distinguished vertical direction is guaranteed; see Definition \ref{def:layer-Wass}.
 The metric barycenter arising from this new metric is a natural candidate for a representative of a family of root systems. Furthermore, we suspect this distance may play a role in other applied problems featuring both tree-like and geometric structures (blood vessels in biology, river systems in topography, etc.).  Our primary present goal is to  develop the mathematical properties of the layerwise-Wasserstein distance and its interpolants, while the biological and methodological applications will be developed in subsequent work. However, we keep the motivating applications in mind as we go, and focus on properties of root systems that have potential biological relevance.

It is common in biology to model root systems by their \emph{skeletons}, in which three dimensional limbs are replaced by approximating one
dimensional curves  \cite{Bucksch14};    these skeletons retain the dendritic, or
treelike, structure of the root, but strip away its thickness (which
is less crucial in some applications). As a corresponding
mathematical object we introduce {\em skeletal measures}, which are essentially
mass distributions supported on these skeletal structures;  see Section~\ref{sec:skeleton}.  This gives
a useful framework for studying the topological properties of roots
and their interpolations, while avoiding difficulties that arise when
dealing with their (more realistic) three dimensional structure.

When building interpolants to use as representatives of families of
roots, a desirable property is that the dendritic structure is
preserved: given several root systems, does their metric barycenter
look like a root?  We are able to give a fairly satisfactory
affirmative answer to this question for skeletal root systems, using
our layerwise-Wasserstein distance as the metric;  see Theorem~\ref{thm:lw-bary}.  On the other hand,
we exhibit examples illustrating that when the conventional
Wasserstein distance is used, interpolants of root systems may not
resemble root systems at all; more precisely, we show that the
Wasserstein barycenter of several skeletal roots can have high
dimensional support, so that the dendritic structure is broken; see Section~\ref{subsect: comparision}.  We also establish comparisons between the total root length (essentially one dimensional Hausdorff measure of the support) of several skeletal root measures and their layerwise-Wasserstein barycenter see Proposition \ref{prop:root-length}; this type of result is impossible in general with the Wasserstein barycenter, as the support may be more than $1$ dimensional.

Aside from being natural for certain applications, the
layerwise-Wasserstein distance also has computational advantages over
its classical Wasserstein counterpart in certain situations, as the
sorting in the distinguished direction is monotone, and so
optimization problems arise only in spaces of co-dimension $1$. In
$\mathbb{R}^2$, for instance, the layerwise-Wasserstein distance
essentially corresponds to the Knothe-Rosenblatt rearrangement
\cite{Knothe57, Rosenblatt52}, which can  be computed much more easily
than the two dimensional Wasserstein distance; however, to the best of
our knowledge, the Knothe-Rosenblatt rearrangement has not been
associated with a metric before, although it has been connected to
optimal transport in
\cite{CarlierGalichonSantambrogio09}).\footnote{Interpolating between
  two dimensional measures is in fact not merely a mathematical
  simplification or toy model, but has actual agricultural
  applications, since experiments are sometimes done growing plants
  between two panes of glass, placed very close together, resulting in
  essentially two dimensional root shapes.}   More generally, the layerwise-Wasserstein distance is a special instance of the Monge-Knothe maps recently introduced in \cite{MuzellecCuturi19}; in that work, properties of the corresponding metric, including interpolation between measures and convexity were not studied.

We also note that our layerwise-Wasserstein distance is similar in
spirit to the Radon-Wasserstein distance found in
\cite{BonneelRabinPeyrePfister2015}, as both approaches involve
disintegrating the measures and transporting their fibres. The
difference lies in how the measures are disintegrated;  we
disintegrate with respect to a distinguished, vertical variable on the
underlying space (which is natural in the applications we have in
mind), whereas the disintegration in
\cite{BonneelRabinPeyrePfister2015} is done with respect to Radon
transformed variables.

 In addition, it is worth commenting briefly on the relationship between this work and another recent series of papers relating optimal transport to plant root shapes \cite{BressanSun2018, BressanPalladinoSun2018}.  In those works, the objective is to identify and characterize root (and tree) shapes which optimize certain functionals, modeling absorbtion of nutrients and sunlight and the cost (via ramified optimal tranpsort) of returning those nutrients to the base of plant, whereas our goal is to differentiate and interpolate between various root systems.

The manuscript is organized as follows.  In Section~\ref{sec:layer},
we introduce the layerwise Wasserstein distance and barycenters, and
establish some basic properties.  Section~\ref{sec:skeleton} focuses
on skeletal measures, while Section~\ref{sec:displacementconvexity} is
devoted to layerwise displacement interpolation and convexity.

\section{Layerwise Wasserstein distance}\label{sec:layer}


Let $M(X)$, respectively $P(X)$, denote the space of  finite
Borel measures, respectively, Borel probability measures, on a metric
space $X$ equipped with the weak-* topology.  Consider $M(\R^d \times
\R_{\ge 0})$ and let $M_{ac}(\R^d \times \R_{\ge 0})$ be its subset
consisting of absolutely continuous measures (with respect to
Lebesgue).  For $\mu \in M(\R^d \times \R_{\ge 0})$, let $\mu^V$ be
its vertical marginal, defined by, $$\int_{\R_{\ge 0}} f (z) \mu^V (dz) =
\int_{\R^d \times \R_{\ge 0}} f(z) \mu (dx, dy), \forall f \in
C(\R_{\ge 0}).$$ Note that $|\mu^V| = |\mu|$,  where $|\mu|$
  denotes the total mass of $\mu$. The following vertical rescaling
of the measures in $M(\R^d \times \R_{\ge 0})$ is a key step in our
construction of the Wasserstein type distance that uses the
distinguished coordinate $\R_{\ge 0}$. Note also that measures may not
necessarily have the same mass, so we also normalize them to be
probability measures.

\begin{definition}[vertical rescaling]
  Given $\mu \in M(\R^d \times \R_{\ge 0})$, we define its  {\em vertically rescaled version}, namely, $$\tilde \mu \in P(\R^d \times [0,1]),$$ as follows:
  Let $F_\mu : \R_{\ge 0} \to [0,1]$ be the cumulative function  given by 
  \begin{align*}
    F_\mu ( y) = \frac{1}{|\mu|}\mu^V( [0,y]) .
  \end{align*}
  Note that $(F_\mu)_\# \mu^V =|\mu| \mathcal{L}^1$, and  $F_\mu$ is continuous for absolutely continuous $\mu^V$. 
  
  Then, define
  \begin{align*}
    \tilde \mu = \frac{1}{|\mu|}(id \times F_\mu)_\# \mu
  \end{align*}
  where $id: \R^d \to \R^d$ is the identity map. 
  Notice that the map $\mu \mapsto \tilde \mu$  from $M(\R^d \times \R_{\ge 0})$ to $P(\R^d \times [0,1])$  is continuous with respect to the weak* topology. In particular, this map pushes forward a given $\Omega \in P(M(\R^d \times \R_{\ge 0}))$, to its vertically rescaled version $$\tilde \Omega \in P(P(\R^d \times [0,1])).$$
\end{definition}
Note that the mapping $F_\mu$ depends on $\mu$ only through its vertical marginal, $\mu^V$; we will sometimes abuse notation and write $F_{\mu^V}$ instead of $F_\mu$.

This normalization allows us to define a Wasserstein type distance
that uses the disintegration along the vertical line.  In the
  following, $W_2^2$ denote the quadratic Wasserstein distance.
\begin{definition}[layerwise-Wasserstein distance]\label{def:layer-Wass}
  Given $\mu, \nu \in M(\R^d \times \R_{\ge 0})$, define
  \begin{align}\label{eqn:LW}
    d_{LW}^2 (\mu, \nu) = W_2^2 \left(\frac{1}{|\mu^V|} \mu^V, \frac{1}{|\nu^V|}\nu^V  \right)  + \int_{0}^{1} W_2^2 (\tilde \mu_l, \tilde \nu_l) dl
  \end{align}
  where $\tilde \mu$ and $\tilde \nu$ have disintegrations $\tilde \mu (dx, dl) = \tilde \mu_l  (dx) dl ,  \tilde \nu (dx, dl) = \tilde \nu_l  (dx) dl$ with respect to the Lebesgue measure $dl$ on $[0,1]$. 
\end{definition}

\begin{remark}
 We note that strictly speaking $d^2_{LW}$ does not give a distance on $M(\R^d \times \R_{\ge 0})$, unless restricted to   $P(\R^d \times \R_{\ge 0})$,  as different measures may have the same vertical rescaling $(\frac{1}{|\mu^V|} \mu^V, \tilde \mu$);  instead, it gives a metric on the set of equivalence classes, under the equivalence  relation $\mu \sim \nu$ if $\mu/|\mu| = \nu/|\nu|$. To get a distance on $M(\R^d \times \R_{\ge 0})$ one may add $(|\mu| -|\nu|)^2$ and consider the metric
  \begin{align*}
   W_2^2 \left(\frac{1}{|\mu^V|} \mu^V, \frac{1}{|\nu^V|}\nu^V  \right)  + \int_{0}^{1} W_2^2 (\tilde \mu_l, \tilde \nu_l) dl +(|\mu| -|\nu|)^2.
  \end{align*}
  In the following, however, we stick to \eqref{eqn:LW} for simplicity (in fact, in subsequent sections, we restrict our attention entirely to $P(\R^d \times \R_{\ge 0})$).
\end{remark}

 We now consider  the metric barycentre  corresponding to the layerwise-Wasserstein distance \eqref{eqn:LW},\footnote{Strictly speaking, given the remark above, the metric barycenter is an equivalence class of measures; we choose as a representative the unique probability measure in a given class.}	which we define below, and call them layerwise-Wasserstein barycentre.

\begin{definition}[layerwise Wasserstein barycentre]
  For $\Omega \in P(M(\R^d \times \R_{\ge 0}))$, a {\em layerwise Wasserstein barycentre} $Bar^{LW} (\Omega) \in P(\R^d \times \R_{\ge 0})$ is defined as an element of 
  \begin{align*}
    Bar^{LW} (\Omega) \in \mathop{\rm argmin}_{\mu \in P\left({\R^d \times \R_{\ge 0} }\right)}  \int_{ M(\R^d \times \R_{\ge 0})}d_{LW}^2 (\mu, \nu) d\Omega (\nu).
  \end{align*}
\end{definition}
To characterize layerwise-Wasserstein barycenters, we need a little more terminology.
Define $\tilde \Omega_l :=\Big(\nu \mapsto \tilde \nu_l \Big)_\#\Omega$.  A Wasserstein barycenter of $\tilde \Omega_l$ is then a minimizer over $P(\mathbb{R}^d)$ of
\begin{equation}\label{eqn: Wass barycenter}
\eta \mapsto \int_{P(\mathbb{R}^d)} W_2^2(\eta, \alpha)d\tilde \Omega_l(\alpha) =\int_{ M(\R^d \times \R_{\ge 0})} W_2^2(\eta, \tilde \nu_l)d\Omega(\nu).
\end{equation}
Similarly, defining $\Omega^V :=\Big(\nu \mapsto  \nu^V \Big)_\#\Omega$, a Wasserstein barycenter of $\Omega^V$ is a minimizer over $P(\mathbb{R}_{\geq 0})$ of
$$
\eta \mapsto \int_{P(\mathbb{R}_{\geq 0})} W_2^2(\eta, \alpha) d\Omega^V(\alpha) =\int_{ M(\R^d \times \R_{\ge 0})} W_2^2(\eta, \nu^V)d\Omega(\nu).
$$
We then have the following:
\begin{proposition}\label{prop:decomp}
	A measure $\mu \in P(\R^d \times \R_{\ge 0})$ is a layerwise Wasserstein barycenter of $\Omega \in P (M(\R^d \times \R_{\ge 0}))$ if and only if its vertical marginal $\mu^V$ is a Wasserstein barycenter of $\Omega^V$ and for almost every layer $l$, $\tilde \mu_l$ is a Wasserstein barycenter of $\tilde \Omega_l$.
	\end{proposition}
\begin{proof}
By definition, a layerwise Wasserstein barycenter $\mu$ must minimize
\begin{eqnarray*}
& \int_{M(\R^d \times \R_{\ge 0})}	\left[ W_2^2 \left(\frac{1}{|\mu^V|} \mu^V, \frac{1}{|\nu^V|}\nu^V  \right)  + { \int_0^1}  W_2^2 (\tilde \mu_l, \tilde \nu_l) dl\right]d\Omega(\nu)\\
&=\int_{M(\R^d \times \R_{\ge 0})}	 W_2^2 \left(\frac{1}{|\mu^V|} \mu^V, \frac{1}{|\nu^V|}\nu^V  \right)d\Omega(\nu)  +\int_{M(\R^d \times \R_{\ge 0})} { \int_0^1} W_2^2 (\tilde \mu_l, \tilde \nu_l) dld\Omega(\nu)\\
&=\int_{M( \R_{\ge 0})}	 W_2^2 \left(\frac{1}{|\mu^V|} \mu^V, \frac{1}{|\nu^V|}\nu^V  \right)d\Omega^V(\nu^V)  + { \int_0^1} \int_{M(\R^d)}W_2^2 (\tilde \mu_l, \tilde \nu_l) d\tilde\Omega_l(\tilde \nu_l)dl.
\end{eqnarray*}
By changing  $\mu^V$ and $\tilde \mu_l$ independently, we see that   $\mu$ minimizes the last line if and only if its vertical marginal $\mu^V$ minimizes the first term and for almost every $l$, $\tilde \mu_l$ minimizes $ \int_{M(\R^d)}W_2^2 (\tilde \mu_l, \tilde \nu_l) d\tilde\Omega_l(\tilde \nu_l)$; that is, $\frac{\mu^V}{|\mu^V|}$  is a Wasserstein barycenter of $\Omega^V$ and $\tilde \mu_l$ a Wasserstein barycenter of $\tilde \Omega_l$.
\end{proof}

The proposition gives a straightforward way to construct
layerwise-Wasserstein barycenters; first construct the layers $\tilde
\mu_l = Bar^W(\tilde \Omega_l)$, as Wasserstein barycenters of the
$\tilde \Omega_l$.  Then letting $\mu^V=Bar^W (\Omega^V)$ be the
Wasserstein barycenter of $\Omega^V$ the layerwise Wasserstein
barycenter $\mu =Bar^{LW} (\Omega)$ is defined by
$$
d\mu(x,y) = d\tilde \mu_{F_{\mu^V}(y)}(x)d\mu^V(y). 
$$
Note that any $Bar^{LW} (\Omega)$ is written this way, and is uniquely determined if $\tilde \mu_l$ is uniquely determined for a.e. $l$. In particular, we have 

\begin{corollary}\label{thm:layerwise-Wass}
	For $\Omega \in P (M_{ac}(\R^d \times \R_{\ge 0}))$, there is unique $\mu=Bar^{LW} (\Omega) $. 
\end{corollary}
\begin{proof}
As $\Omega \in P (M_{ac}(\R^d \times \R_{\ge 0}))$, it also holds that  $\tilde \Omega_l \in P (M_{ac}(\R^d))$ for a.e. $l$.  Then uniqueness of $\tilde \mu_l$ follows from \cite{KimPass17}. 
\end{proof}

 The rescaled version $ \tilde Bar^{LW}(\Omega) \in P(\mathbb{R}^d \times
   [0,1])$ of the layerwise-Wasserstein barycenter $Bar^{LW}(\Omega)$ has the disintegration $ d\tilde
  Bar^{LW}(\Omega)(x,l)=d \tilde Bar^{LW}_l(\Omega) dl$, where each
  $Bar^{LW}_l(\Omega)$ is a Wasserstein barycenter of the $\tilde
  \Omega_l$.  The rescaling mapping $F_{Bar^{LW}(\Omega)}$ satisfies
\begin{align}\label{eqn:F-LW}
   F_{Bar^{LW}(\Omega)}(y) =[\int F_\nu^{-1}d\Omega(\nu)]^{-1}(y).
\end{align}

We note here associativity, in the two dimensional case, i.e. on $\R
\times \R_{\ge 0}$, of the layerwise-Wasserstein barycenter of
probability measures $\mu_1,...,\mu_N$ with weights
$\lambda_1,...,\lambda_N$, where $\sum_{i=1}^N\lambda_i=1$ and each
$\lambda_i \geq 0$.
\begin{proposition}(Associativity of $2$-dimensional layerwise-Wasserstein barycenters)
  Assume that $d=1$ and $\lambda_1 +\lambda_2 +\lambda_3 =1$.  Then
  \begin{align*}
    & Bar^{LW}(\lambda_1\delta_{\mu_1}+\lambda_2\delta_{\mu_2}+\lambda_3\delta_{\mu_3}) \\
    &=Bar^{LW}\Big((\lambda_1+\lambda_2)\delta_{Bar^{LW}(\frac{\lambda_1}{\lambda_1+\lambda_2}\delta_{\mu_1}+\frac{\lambda_2}{\lambda_1+\lambda_2}\delta_{\mu_2})} +\lambda_3\mu_3\Big).
  \end{align*}
\end{proposition}
This proposition is potentially useful in certain computations, as when one
adds a new sample $\mu_{N+1}$ root system to a family of $N$ root
systems with a (previously computed) barycenter $\bar \mu$, one can
find the barycenter of the augmented family by computing the
appropriately weighted barycenter of $\mu_{N+1}$ and $\bar \mu$,
rather than the more difficult computation of the barycenter of the
new family of $N+1$ systems.
\begin{proof}
  The result follows immediately from the corresponding result in one dimension for Wasserstein barycenters.
\end{proof}
\begin{remark}
  In our motivating application, we only distinguish between root
  systems up to rotation about the vertical axis; that is, we wish to
  identify two systems whenever we can transform one system to the
  other via a rotation fixing $y$.  For actual root systems then, the
  following distance is relevant:
  \begin{definition}[Horizontally symmetrized layerwise-Wasserstein distance]We define the horizontally symmetrized layerwise-Wasserstein distance $d_{LW, symm}^2 (\mu, \nu)$ between $\mu$ and $\nu$ by
    $$
    d_{LW, symm}^2 (\mu, \nu) = \min_{R\in SO(d)}d_{LW}^2 (R_\#\mu, \nu),
    $$
where $SO(d)$ denotes the special orthogonal group on the horizontal directions $\mathbb{R}^d$. 
  \end{definition}

 Note that $d_{LW, symm}$ is a metric on the set of equivalence
  classes of probability measures under horizontal rotational equivalence (that
  is, $\nu \sim \mu$ if $\nu =R_\#\mu$ for some rotation $R \in
  SO(d)$).  A horizontally symmetrized Wasserstein barycenter
  $Bar^{LW}_{symm}(\Omega)$ of a measure $\Omega \in P (M(\R^d \times
  \R_{\ge 0}))$ is then a metric barycenter with respect to this
  distance; that is, a minimizer of:
$$
  \nu \mapsto \int_{M(\R^d \times \R_{\ge 0})}d_{LW,symm}^2(\nu,\mu)d\Omega(\mu).
$$
Equivalently, $Bar^{LW}_{symm}(\Omega)$ minimizes 
$$
\nu \mapsto \min_{R_\mu \in SO(d) \forall \mu \in P(M)}\int_{M(\R^d \times \R_{\ge 0})}d_{LW}^2(\nu,(R_\mu)_\#\mu)d\Omega(\mu).
$$
Analogously, one could also consider rotationally symmetrized versions of the classical Wasserstein distance:
$$
W^2_{2, symm}(\mu,\nu) := \min_{R\in SO(d)} W_2^2(\mu,R_\#\nu)
$$
and corresponding barycenters, which are minimizers of:
\begin{equation}\label{eqn: symmetrized wass bc}
\nu \mapsto \min_{R_\mu \in SO(d) \text{ }\forall\mu \in P(M)}\int_{P(\R^d \times \R_{\ge 0})}W_{2}^2(\nu,(R_\mu)_\#\mu)d\Omega(\mu).
\end{equation}

Symmetrized Wasserstein barycenters are more natural for the root
interpolation problem than classical Wasserstein barycenters.  One of
our goals in this paper is to demonstrate that symmetrized
layerwise-Wasserstein barycenters are better suited for this problem
than classical (symmetrized or unsymmetrized) Wasserstein barycenters;
to this end, we provide examples in Section \ref{subsect:
  comparision} of measures $\mu_1, ... \mu_m$ which are root like in a
certain sense (skeletons in the nomenclature of the next section), for
which the symmetrized Wasserstein barycenter of
$\frac{1}{m}\sum_{i=1}^m\delta_{\mu_i}$ does not resemble a root (that
is, is not a skeleton). Their layerwise-Wasserstein barycenter, on the
other hand, has a much more root like structure (see Theorem
\ref{thm:lw-bary} below).  
\end{remark}
\section{Skeletal measures}\label{sec:skeleton}

Real plant root systems consist of limbs with thickness.  However,
biologist often approximate roots by their "skeletons," in which each
limb is replaced by a one dimensional curve, thus retaining the
topological, or detritic structure of the root, but losing its
thickness.  Below, we provide a formal mathematical definition of
skeletons, and introduce skeletal measures, which are essentially
distributions of mass supported on them.

\begin{definition}\label{def:skeleton}
Let $Y =[0,\bar y] \subset \R$, be an interval whose length $\bar y$,
represents the vertical depth of the root.  A {\em weak skeletal root}
consists of the graphs of a finite union of curves, $$\bigcup_{i=1}^N
{\rm graph}(g_i),$$ where each $g_i: [\underline y_i, \overline y_i]
\rightarrow \mathbb{R}^d$ is a Lipschitz function defined on a
subinterval $[\underline y_i, \overline y_i] \subseteq Y$, satisfying
the following properties:
\begin{itemize}
\item[\bf S1] (Roots start from a common stem) $\underline y_1=0$ and $\underline y_i >0$ for each $i=2,...N$.
\item[\bf S2] (Limbs emerge from older limbs) For each $i=2,....N$, there is some $j <i$ such that $\underline y_i \in (\underline {y}_j, \overline y_j)$ and $g_i(\underline y_i) = g_j(\underline y_i)$.

\end{itemize}
A {\em strong skeletal root} is a weak skeletal root which satisfies the additional condition:
\begin{itemize}
\item[\bf S3] (Limbs never cross each other)  For each $i \neq j$ and all $y \in (\underline y_i, \overline y_i] \cap (\underline y_j, \overline y_j]$, we have $g_i(y) \neq g_j(y)$.
\end{itemize}
\end{definition}
We next define strong skeletal root measures.
\begin{definition}
  A \emph{strong skeletal root measure} is a probability measure whose
  support is an \emph{entire} strong skeletal root, which is
  absolutely continuous with respect to the one dimensional Hausdorff
  measure.
\end{definition}
 
Strong skeletal root measures seem to be reasonable proxies for real
roots. 
As we will see
below, layerwise-Wasserstein barycenters of strong skeletal root
measures preserve the one dimensional structure of the support (this
is an important distinction from conventional Wasserstein barycenters
-- see Example \ref{ex: high dim barycenter} below).  Unfortunately,
they are not always strong skeletal root measures, for two reasons: 1)
the support may be disconnected, and 2) The non-crossing property
holds only in a weaker sense.  This motivates the following
definition:

\begin{definition}
A \emph{weak skeletal root measure} is a probability measure supported
on a weak skeletal root, which is absolutely continuous with respect
to the one dimensional Hausdorff measure, satisfying the following
additional property:
\begin{itemize}
\item[\bf W3]  For each $i \neq j$ and all $y \in (\underline
  y_i, \overline y_i] \cap (\underline y_j, \overline y_j]$, such that
  $g_i(y) =g_j(y)$, we have either $\lim_{z \rightarrow y^-}
  \mu_{z}(\{g_i(z)\}) =0$ or $\lim_{z \rightarrow y^-}
  \mu_{z}(\{g_j(z)\}) =0$.
\end{itemize}
 where $\mu_y =\tilde \mu_{F_\mu(y)}$ is the conditional probability of $d\mu(x,y) =d\mu_y(x)d\mu^V(y)$.
\end{definition}
\noindent Note that by construction,  for each  $l$, the layer $\tilde \mu_l$ of a weak skeletal root measure $\mu$, is a convex combination of Dirac masses. 

Obviously strong skeletal root measures are weak skeletal root measures; weak skeletal root measures are essentially ``roots with missing parts," and have a weaker version {\bf W3}  of the no crossing condition. While the layerwise-Wasserstein barycenter of several strong skeletal root measures may not be a strong skeletal root measure, we are able to show below that it is a weak skeletal root measure. 
\begin{remark}
Interpreting each $graph(g_i)$ as a limb, condition {\bf S3} expresses the natural expectation that  limbs do not cross.  We interpret {\bf W3} as a weaker version of this: if $g_i(y) =g_j(y)$ and  $\lim_{z \rightarrow y^-} \mu_{z}(\{g_i(z)\}) =0$, we interpret $g_i$ as consisting of two limbs: an upper limb $g_i^1$, defined by restricting $g_i$ to $[\underline y_i, y]$, and a lower limb, $g_i^2$, obtained by restricting $g_i$ to $[y,\overline y_i]$, emerging from the older limb $y_j$ at the point $y$.  This seems reasonable to us, since the hypothesis $\lim_{z \rightarrow y^-} \mu_{z}(\{g_i(z)\}) =0$ means that there is no mass at $y$ coming from the upper limb; the upper limb thus ends at the point $y$. 

By interpreting a weak root as a tree in this sense, one can compute topological properties which are defined only for loop-free structures (including, for example, the Horton-Strahler index \cite{Toroczkai01}, often used by biologists to measure the topological complexity of root systems).

\end{remark}

\begin{remark}
 Skeletons can be computationally useful in practice.  Algorithms are
 available to construct skeletons from real root system data,
 essentially by tracing the center of mass of the cross sections of
 each limb \cite{Bucksch14}.  Computing layerwise-Wasserstein barycenters of these
 skeletons is then much less computationally intensive than computing
 the barycenters of the original roots, since each layer is
 discretized by many fewer points, but may still provide valuable
 biological insight about the "average" topological structure of the
 family of root systems.
\end{remark}

 Assuming that $\mu$ is a (weak or strong, respectively)
  skeletal root measure, supported on the skeletal root
  $\bigcup_{i=1}^N {\rm graph}(g_i),$ and the rescaling map $F_\mu$ is
  bi-Lipschitz, $\tilde \mu$ is also a (respectively weak or strong)
  skeletal root measure, supported on the skeletal root
  $\bigcup_{i=1}^N {\rm graph}(\tilde g_i),$ where $\tilde g_i :=
  g_i\circ F_\mu^{-1}$.  Note that the domain $[\underline l_i,
    \overline l_i]:=[F_\mu (\underline y_i), F_\mu(\overline y_i)]$ of
  each rescaled limb $\tilde g_i$ is contained in $[0,1]$.  We call
  $\bigcup_{i=1}^N {\rm graph}(\tilde g_i)$ a \textit{rescaled}
  skeletal root.

\subsection{Layerwise Wasserstein barycenters of skeletal root measures}\label{sec:LW-skel}
We now prove that layerwise-Wasserstein barycenters of weak skeletal
root measures are themselves weak-skeletal root measures. 
\begin{theorem}\label{thm:lw-bary}
	Let $\mu_1,....,\mu_m \in P( \mathbb{R}^d \times \R_{\ge 0})$ be compactly supported  weak skeletal root measures such that $l\mapsto (\tilde \mu_i)_l$ is weak-$*$ continuous and
	$F_{\mu_i}$ is bi-Lipschitz for each $i$, and $\lambda_1,....,\lambda_m >0$ with $\sum_{i=1}^m\lambda_i=1$.
	Then any layerwise-Wasserstein barycenter $Bar^{LW}(\sum_{\alpha} \lambda_\alpha \mu_\alpha)$ of $\mu_1,....,\mu_m$ with weights $\lambda_1,....,\lambda_m$ is also a weak skeletal root measure.
\end{theorem}


\begin{remark}

We expect this result to play an important role in biological
applications.  As mentioned above, given a family of root systems, we will propose in
future work interpreting the layerwise-Wasserstein barycenter as the
best representative of that family.  It is therefore desirable to
compute certain biologically relevant traits of the barycenter,
especially those traits that rely on its dendritic structure,   for
instance the total root length and the Horton-Strahler (HS) index \cite{Toroczkai01}.
The HS index in particular relies on the non crossing property, 
can be defined for weak skeletal root measures, thanks to {\bf W3},
but not for more general unions of graphs such as weak skeletal roots.
\end{remark}

The key tool in the proof of this theorem is the barycentric ghost, which we define now. 
\begin{definition}\label{def:ghost}
  For $\alpha =1,2,...,m$, let $\tilde S_\alpha:=\{\tilde
  g_{i_\alpha}^\alpha: i_\alpha=1,2,...N_\alpha \}$ be a  rescaled
  skeletal root, and let $\lambda = (\lambda_1,\lambda_2,...,\lambda_m)$ with $\lambda_1,....,\lambda_m >0$ be a collection of 
  weights with $\sum_{\alpha=1}^{m}\lambda_\alpha=1$.
  
  For fixed indices $i_1,...,i_m$, whenever the intersection
  $\cap_{\alpha=1}^m[\underline l^\alpha_{i_\alpha},\overline
    l^\alpha_{i_\alpha} ]$ of domains $[\underline
    l^\alpha_{i_\alpha},\overline l^\alpha_{i_\alpha} ]$ of the
family  $\{ \tilde g^\alpha_{i\alpha}\}$ is non-empty, we define the curve
  $$
  \tilde G^\lambda_{i_1i_2,...,i_m}:=\sum_{\alpha=1}^m\lambda_\alpha \tilde g^\alpha_{i_\alpha} .
  $$	
  The {\em ghost} of the family  $\{\tilde S_\alpha\}$ with weights $\lambda$ is then the collection of curves $\tilde G^\lambda_{i_1i_2,...,i_m}$.
  
\end{definition}
At each slice $l \in [0,1]$, the set $\tilde
G^\lambda_{i_1i_2,...,i_m}(l)$ represents the Euclidean barycenters of
all possible combinations of $\tilde g^\alpha_{i_\alpha}(l) $ in the
supports of the discrete sliced layers. The Wasserstein barycenter of
the layers is supported on these points, therefore, for skeletal root
measures $\mu_1,...,\mu_m$, supported respectively on
$\bigcup_{i_\alpha=1}^{N_\alpha} {\rm graph}(g^\alpha_{i_\alpha})$ we have the
following:
\begin{align}\label{eqn:ghostsupports}
&\hbox{If $(x,y)\in {\rm supp} \left(Bar^{LW}(\sum_{\alpha =1}^m \lambda_\alpha \delta_{\mu_\alpha})\right) $}, 
\\\nonumber
&\hbox{ then $x=\tilde G^\lambda_{i_1i_2,...,i_m}\Big((\sum_{\alpha =1}^{m}\lambda_\alpha F_{\mu_\alpha}^{-1})^{-1}(y)\Big)$ for some choice of $i_1,....,i_m$. }
\end{align}

Given probability root measures, the ghost of their rescaled supports  $\bigcup_{i_\alpha=1}^N {\rm graph}(\tilde g^\alpha_{i_\alpha})$ can be un-rescaled via the mapping $y \mapsto (\sum_{\alpha=1}^m \lambda_\alpha F_{\mu_\alpha}^{-1})^{-1}(y)$; the un-rescaled ghost is then the union of the graphs $G^\lambda_{i_1i_2,...,i_m}(y):=\tilde G^\lambda_{i_1i_2,...,i_m}  \Big((\sum_{\alpha=1}^m \lambda_\alpha F_{\mu_\alpha}^{-1})^{-1}(y)\Big)$.

It is then easy to see that the layerwise-Wasserstein barycenter has support contained in the (un-rescaled) ghost, though it typically won't fill it out.  We think of the ghost sitting in the background; it is the largest possible potential support of the barycenter.  We think of the actual support of the barycenter as sitting in the foreground on top of it.

Now, the ghost clearly satisfies {\bf S1} (starting as stem) and {\bf
  S2} (limbs emerge from older limbs) in the definition of skeletal
roots.  It does not generally satisfy {\bf S3} (non crossing).  In
order to verify that the layerwise-Wasserstein barycenter is a weak
skeletal root measure, we must therefore show that it satisfies the
weak non-crossing property {\bf W3}.

The following Lemma essentially verifies {\bf W3} for the rescaled
barycenter; since it is clear that the bi-Lipschitz rescaling
$\sum_{\alpha =1}^{m}\lambda_\alpha F_{\mu_\alpha}^{-1}$, which pushes
$\tilde Bar^{LW}(\sum_{\alpha} \lambda_\alpha \mu_\alpha)$ forward to
$Bar^{LW}(\sum_{\alpha} \lambda_\alpha \mu_\alpha)$ preserves this
property, the  lemma implies Theorem \ref{thm:lw-bary}.

\begin{lemma}\label{lem:W3}
Under the same assumptions as in Theorem~\ref{thm:lw-bary}, let $\mu=Bar^{LW}(\sum_{\alpha} \lambda_\alpha \mu_\alpha)$ be a layerwise-Wasserstein barycenter of the $\{\mu_\alpha\}$'s.  
Set $l=(\sum_{\alpha =1}^{m}\lambda_\alpha F_{\mu_\alpha}^{-1})^{-1}(y)$, and 
	suppose $x =\tilde G^\lambda_{j_1j_2,...,j_m}(l) =\tilde G^\lambda_{i_1i_2,...,i_m} (l)$, where $j_\alpha \neq i_\alpha$ for at least one $\alpha$ and $l$ is not the minimal point in the domain of $\tilde G^\lambda_{i_1i_2,...,i_m}$ or $\tilde G^\lambda_{j_1j_2,...,j_m}$. 
 Moreover, suppose that $\lim_{z \to l^-} \tilde \mu_{z}(\tilde G^\lambda_{i_1i_2,...,i_m}(z))> 0$. 
	Then, 
	$$
	\lim_{z \rightarrow l^-}\tilde \mu_{z}(\{\tilde G^\lambda_{j_1j_2,...,j_m}(z)\} \setminus \{G^\lambda_{i_1i_2,...,i_m}(z)\} ) =0.  
	$$	
\end{lemma}
The proof of this lemma leverages a connection between the Wasserstein barycenter of $\sum_{\alpha}\lambda_\alpha \delta_{(\tilde \mu_\alpha)_l}$ and the multi-marginal extension of optimal transport, which is to minimize
\begin{equation}\label{eqn: multi-marginal}
\int_{(\R^d )^m}\sum_{\alpha ,\beta}\lambda_\alpha\lambda_\beta|x_\alpha-x_\beta|^2d\gamma(x_1,x_2,...,x_m)
\end{equation}
among all probability measures $\gamma$ on $(\R^d)^m$ whose marginals are the $(\tilde \mu_\alpha)_l$. It is well known that  
the mapping 
\begin{align*}
\Delta_\lambda:  (x_1,x_2,...,x_m) \rightarrow \sum_\alpha\lambda_\alpha x_\alpha
\end{align*}
pushes each solution $\tilde \gamma_l $ forward to a Wasserstein barycenter $\tilde \mu_l$ \cite{AguehCarlier11}, and this mapping is invertible with a Lipschitz inverse on the support of $\tilde \mu_l$; see e.g. \cite{KimPass17}. 
\begin{proof}[Proof of Lemma~\ref{lem:W3}]

For each layer $l$, we will denote by $\tilde \gamma_l  \in P(\R^d \times \cdots \times \R^d)$  a solution to the multi-marginal optimal transport problem \eqref{eqn: multi-marginal}. 
 Assume that the conclusion of the lemma fails. Then there exists $\epsilon >0$ and a sequence $l_k <l$ converging to $l$ such that $\tilde G^\lambda_{j_1j_2,...,j_m}(l_k) \neq \tilde G^\lambda_{i_1i_2,...,i_m}(l_k)$ and  $\tilde \mu_{l_k}(\tilde G^\lambda_{j_1j_2,...,j_m}(l_k))>\epsilon$, for large enough $k$.  
	
	 We first prove the lemma under the simplifying assumption that  
$\tilde G^\lambda_{j_1j_2...j_m}(l_k) \neq \tilde G^\lambda_{j'_1j'_2...j'_m}(l_k)$ for all $(j'_1,j'_2,...,j'_m) \neq (j_1,j_2,...,j_m)$.

This immediately implies that 
\begin{equation}\label{eqn: multi-marg charges point}
\tilde \gamma_{l_k}(\tilde g^1_{j_1}(l_k),\tilde g^2_{j_2}(l_k),...,\tilde g^m_{j_m}(l_k)) >\epsilon 
\end{equation}
and in particular
\begin{equation}\label{eqn: marginals charge graphs}
 (\tilde \mu_\alpha)_{l_k}(g^\alpha_{j_\alpha}(l_k)) >\epsilon,
\end{equation}
for $\alpha=1,2....,m$.
  After passing to a convergent subsequence, the $\tilde \gamma_{l_k}$ converge (in the weak-$*$ sense) to a measure $\tilde \gamma_l$ which is optimal in the multi-marginal problem for the $(\tilde \mu_{i_\alpha})_{l}$, and 
  $$\tilde \gamma_{l}(\tilde g^1_{j_1}(l),\tilde g^2_{j_2}(l),...,\tilde g^m_{j_m}(l)) \neq 0.
  $$  Exactly the same argument implies the existence of a second minimizer $\tilde \gamma'_{l}$ to the multi-marginal problem such that $\tilde \gamma'_{l}(\tilde g^l_{i_1}(l),\tilde g^2_{i_2}(l),...,\tilde g^m_{i_m}(l)) \neq 0.$

Although it is possible that $\tilde \gamma'_{l} \neq \tilde \gamma_{l}$, their linear average $\frac 1 2 \tilde \gamma'_{l} +\frac 1 2 \tilde \gamma_{l}$ is also optimal for the multi-marginal problem and has both \begin{align*}
\hbox{$(\tilde g^1_{i_1}(l),\tilde g^2_{i_2}(l),...,\tilde g^m_{i_m}(l))$ and $(\tilde g^1_{j_1}(l),\tilde g^2_{j_2}(l),...,\tilde g^m_{j_m}(l))$ in its support,}
\end{align*} with the corresponding Wasserstein barycenter $\hat \mu_l = \frac 1 2 \tilde \mu_{l}+\frac 1 2 \tilde \mu'_{l} $. Notice that 
\begin{align*}
 \Delta_\lambda \left((\tilde g^1_{i_1}(l),\tilde g^2_{i_2}(l),...,\tilde g^m_{i_m}(l))\right) =x = \Delta_\lambda \left((\tilde g^1_{j_1}(l),\tilde g^2_{j_2}(l),...,\tilde g^m_{j_m}(l)) \right).
\end{align*}
    Because $\Delta_\lambda$ has a Lipschitz inverse, we have
\begin{align*}
| (\tilde g^1_{i_1}(l),\tilde g^2_{i_2}(l),...,\tilde g^m_{i_m}(l)) - (\tilde g^1_{j_1}(l),\tilde g^2_{j_2}(l),...,\tilde g^m_{j_m}(l)) | 
\le C |x-x| =0.
\end{align*}
Now, letting $\alpha$ be such that $j_\alpha \neq i_\alpha$, the above implies that $ \tilde g^\alpha_{i_\alpha}(l) =\tilde g^\alpha_{j_\alpha}(l)$.  Since $\mu_\alpha$ is a weak root measure, this means that, without loss of generality, $$	\lim_{z \rightarrow l^-}(\tilde \mu_\alpha)_{z}(\{\tilde g^\alpha_{j_\alpha}(z)\}) =0.$$  
This contradicts \eqref{eqn: marginals charge graphs} and completes the proof 
 under the additional assumption.

Now, if the assumption fails,
instead of \eqref{eqn: multi-marg charges point}, we can conclude only that $$\tilde \gamma_{l_k}(\tilde g^1_{j_1'}(l_k),\tilde g^2_{j_2}(l_k'),...,\tilde g^m_{j_m}(l_k')) >\epsilon$$ \emph{for some} $(j_1',....,j_m')$ with  $G^\lambda_{j_1'j_2'...j_m'}(l_k) = \tilde G^\lambda_{j_1j_2...j_m}(l_k)$, and by passing to a subsequence if necessary, we can take it to be the same $(j_1',....,j_m')$ for each $k$.  As above, this implies that
$$\tilde \gamma_{l}(\tilde g^1_{j_1'}(l),\tilde g^2_{j_2'}(l),...,\tilde g^m_{j_m'}(l)) \neq 0,
$$ and since $\tilde G^\lambda_{j_1'j_2'...j_m'}(l_k) = \tilde G^\lambda_{j_1j_2...j_m}(l_k)$, passing to the limit implies $\tilde G^\lambda_{j_1'j_2'...j_m'}(l) = \tilde G^\lambda_{j_1j_2...j_m}(l) =\tilde  G^\lambda_{i_1i_2...i_m}(l)$.  The rest of the proof follows exactly as in the special case above.
\end{proof}

\begin{remark}
	If the solution 	$\tilde \gamma_{l}$ to the
	  multi-marginal problem \eqref{eqn: multi-marginal} in the proof above is \emph{unique}, and the $\mu_\alpha$ are all strong root measures, then more is true: \\	  
	  If $\lim_{z \rightarrow l^-}\tilde \mu_{z}(\{\tilde G^\lambda_{j_1j_2,...,j_m}(z)\})\neq 0 $ we actually have $\tilde \mu_{l}(G^\lambda_{i_1i_2,...,i_m}(z)) =0 $ for $z<l$ sufficiently close to $l$.

	 To see this, note that as above, $\tilde \gamma_{l}(\tilde g^1_{j_1}(l),\tilde g^2_{j_2}(l),...,\tilde g^m_{j_m}(l)) \neq 0$.  Since the root measures are strong, we must have $\tilde g^\alpha_{i_\alpha}(l) \neq \tilde g^\alpha_{j_\alpha}(l) $ for the $\alpha$ such that $i_\alpha \neq j_\alpha$.  For $z<l$ with $z$ close to $l$, \emph{any} solution $\tilde \gamma_z$ to the multi-marginal plan \eqref{eqn: multi-marginal} must be weak-$*$ close to $\tilde \gamma_l$ (by uniqueness) and so must satisfy $\tilde \gamma_{z}(\tilde g^1_{j_1}(z),\tilde g^2_{j_2}(z),...,\tilde g^m_{j_m}(z)) \neq 0$.  If such a solution satisfied $\tilde \gamma_{z}(\tilde g^1_{i_1}(z),\tilde g^2_{i_2}(z),...,\tilde g^m_{i_m}(z)) \neq 0$ as well, we would then have, by the Lipschitz property of $\Delta_\lambda^{-1}$,
	 \begin{align*}
	 & | (g^1_{i_1}(z),g^2_{i_2}(z),...,g^m_{i_m}(z)) - (g^1_{j_1}(z),g^2_{j_2}(z),...,g^m_{j_m}(z)) | 
	 \\
	 &\le C\left|G^\lambda_{i_1i_2,...,i_m}(z)-G^\lambda_{j_1j_2,...,j_m}(z)\right|
	 \end{align*}
	 
	 However, this is impossible since the right hand side tends to $0$ as $z$ tends to $l$, but the left hand side does not (as $\tilde g^\alpha_{i_\alpha}(l) \neq \tilde g^\alpha_{j_\alpha}(l) $ for at least one $\alpha$, as described above).  We conclude that we must have  $\tilde \gamma_{z}(\tilde g^1_{i_1}(z),\tilde g^2_{i_2}(z),...,\tilde g^m_{i_m}(z)) = 0$ for any solution to the multi-marginal problem and all $z<l$ sufficiently close to $l$; therefore, $\tilde \mu_z(G^\lambda_{i_1i_2,...,i_m}(z)) =0$ for any Wasserstein barycenter $\tilde \mu_z$ of the $\tilde \mu_1,...\tilde \mu_m$.

	This applies, for instance, when $d=1$, in which case Wasserstein barycenters are always unique.
\end{remark}

\subsection{Comparison with the Wasserstein barycenter}\label{subsect: comparision}

If we instead use the standard notion of the Wasserstein barycenter to interpolate between several root measures, the barycenter may not be a weak root measure, as the following examples show.

 \begin{example} Several constructions of Santambrogio and Wang \cite{SantambrogioWang16} show that displacement interpolation does not generally preserve convexity of sets.  In one of these, the two marginals measures are concentrated on line segments embedded in $\mathbb{R}^2$, while their displacement interpolant (or Wasserstein barycenter) is supported on a curve $y=f(x)$  with a strict local minimum, where $y$ is the vertical direction (see $\mu_{1/2}$ in section 2 in \cite{SantambrogioWang16}).    In our context, the two line segments constitute simple strong skeletal root measures, whereas the displacement interpolant is not even a weak skeletal root measure (as the two limbs meeting at the minimum point $x_0$ of $f$ violate \textbf{W3}). Note that this is precisely because the angle between the two limbs is greater than $\pi/2$, and so the optimal map is not monotone in the vertical direction.
\end{example} 
Our second example is even less well behaved; here we take three strong root measures for which the Wasserstein barycenter has three dimensional support.
\begin{example}\label{ex: high dim barycenter}
	Consider uniform measure on the mutually orthogonal segments $T:=\{(t,t,t): t \in [0,1]\}$, $R:=\{(r,(\frac{-1+\sqrt{3}}{2})r,(\frac{-1-\sqrt{3}}{2})r): r \in [0,1]\}$ and  $S:=\{(s,(\frac{-1-\sqrt{3}}{2})s,(\frac{-1+\sqrt{3}}{2})s): s \in [0,1]\}$ in $\mathbb{R}^3$.
	
	Since the segments are orthogonal,  the interaction terms $x_\alpha \cdot x_\beta =0$ in the Gangbo-Swiech cost $(\mathbb{R}^3)^3$ (with, say, $\lambda_\alpha =1/m$, $m=3$) $\sum_{\alpha ,\beta}\frac{1}{9}|x_\alpha-x_\beta|^2 =- \sum_{\alpha ,\beta=1}^3\frac{4}{9} x_\alpha\cdot x_\beta +\sum_{\alpha=1}^3\frac{4}{9} |x_\alpha|^2$ vanish.
	
	Therefore, any measure with Lebesgue marginals supported on the product space $T \times R \times S$ is optimal in the multi-marginal problem \eqref{eqn: multi-marginal}.   The pushforward of any such measure $\gamma$ by the mapping \begin{eqnarray*}
(t,t,t),(r,(\frac{-1+\sqrt{3}}{2})r,(\frac{-1-\sqrt{3}}{2})r),(s,(\frac{-1-\sqrt{3}}{2})s,(\frac{-1+\sqrt{3}}{2})s)  \mapsto\\ \frac{(t,t,t)+(r,r,-2r)+(r,(\frac{-1+\sqrt{3}}{2})r,(\frac{-1-\sqrt{3}}{2})r)+ (s,(\frac{-1-\sqrt{3}}{2})s,(\frac{-1+\sqrt{3}}{2})s)}{3}  	\end{eqnarray*}
is a Wasserstein barycenter.  If $\gamma$ is, for example, product measure, this push forward is absolutely continuous with respect to Lebesgue measure on $\mathbb{R}^3$.

	This is certainly not a skeletal measure, and cannot be interpreted as a root in any reasonable way.
\end{example} 
As with the layerwise-Wasserstein distance, one might suggest that horizontal symmetrization of the classical Wasserstein distance is more appropriate for comparing root shapes.  That is, one should consider minimizers of \eqref{eqn: symmetrized wass bc}.
In the preceding example, although the Wasserstein barycenter has three dimensional support, the biologically more relevant horizontally symmetrized version is concentrated on a line segment (since after appropriate rotations, the three sample measures are the same).  Below, we augment the sample measures to produce a horizontally symmetrized Wasserstein barycenter with three dimensional support.
\begin{example}
	Let $\mu_1, \mu_2$ and $\mu_3$ be uniform measures on the respective domains $S_i$ defined by:
	\begin{eqnarray*}
	S_1:=&\{(t,t,t): t \in [0,1+\epsilon]\} \\
	S_2:=&\{(t,t,t): t \in [0,1]\} \cup \{(1+t,1+(\frac{-1+\sqrt{3}}{2})t,1+(\frac{-1-\sqrt{3}}{2})t) : t \in [0,\epsilon]\}\\
	S_3:=&\{(t,t,t): t \in [0,1]\} \cup \{(1+t,1+(\frac{-1-\sqrt{3}}{2})t,1+(\frac{-1+\sqrt{3}}{2})t): t \in [0,\epsilon]\}
	\end{eqnarray*}
	  It is not hard to show that the identity rotation minimizes the Wasserstein distance between $\mu_i$ and $R_\#\mu_j$ among horizontal rotations $R$ for sufficiently small $\epsilon$. 
	  
	  Furthermore, the optimal plans between $\mu_i$ and $\mu_j$ couple the top limbs via the identify mappings and the bottom limbs via product measure (or any other coupling between the bottom limbs -- the solution is non-unique).  Therefore, the measure 
	  \begin{eqnarray*}
\gamma =&	&  (Id \times Id \times Id)_\# (\mu_1 |_{\{(t,t,t): t \in [0,1]\}})\\
&+&(\mu_1 |_{\{(t,t,t): t \in [1,1+\epsilon]\}}) \times \left(\mu_2 |_{\{(1+t,1+(\frac{-1+\sqrt{3}}{2})t,1+(\frac{-1-\sqrt{3}}{2})t) : t \in [0,\epsilon]\}}\right)\\
&  &\qquad \qquad  \times \left(\mu_3 |_{ \{(1+t,1+(\frac{-1-\sqrt{3}}{2})t,1+(\frac{-1+\sqrt{3}}{2})t): t \in [0,\epsilon]\}}\right)
	  \end{eqnarray*}
	  is optimal in the multi-marginal problem \eqref{eqn: multi-marginal}, and this plan has minimal cost among all multi-marginal problems with marginals $(\mu_1,R_{2\#}\mu_2, R_{3\#}\mu_3)$ for horizontal rotations $R_2$ and $R_3$.  Consequently the symmetrized Wasserstein barycenter from \eqref{eqn: symmetrized wass bc} is then the pushforward of this measure under the mapping $(x_1,x_2,x_3) \mapsto \frac{x_1+x_2+x_3}{3}$; this consists of 
	  the uniform measure on $\{(t,t,t): t \in [0,1]\}$ and a  measure constructed as in the previous example, with three dimensional support, arising from coupling the three orthogonal lower limbs.
\end{example}

\subsection{Total Root Length}  An important phenotype used by biologists to compare root systems is the total root length, which is well defined for skeletal root systems.  
 
Given a strong skeletal root measure $\mu$ supported on $\cup_{i=1}^{N}{\rm graph}(g_{i})$ on $[0,\bar y]$,  for $\alpha =1,2,...,m$,  the total root length of $\mu$  is simply the one dimensional Hausdorff measure of its support. Letting $\chi_i$ be the indicator function of the domain $[\underline y_i,\overline y_i] \subseteq [0,\bar y ]$ of $g_i$, we note that the root length is
\begin{equation}\label{eqn: root length}
R(\mu) =\sum_{i=1}^{N} \int_0^{\bar y} \sqrt{1+| (g_i)'(y)|^2}\chi_i(y) dy.
\end{equation}

Here we establish a result comparing the total root lengths of several skeletal root systems and their layerwise-Wasserstein barycenter. Given strong skeletal root measures 
\begin{align*}
 \hbox{$\mu_\alpha$ supported on $\cup_{i=1}^{N_{\alpha}}{\rm graph}(g^\alpha_{i_\alpha})$ on $[0,\bar y^\alpha]$,  for $\alpha =1,2,...,m$,}
\end{align*}
 we compare their total root lengths to that of (a selected) layerwise-Wasserstein barycenter, with weights $\lambda_1,...\lambda_m$.  As above, we will also assume two sided bounds, $$0< L\leq f_\alpha^V(y) \leq U< \infty,$$ on each $\mu_\alpha$, where $f_\alpha^V$ is the density of the vertical marginal $\mu_\alpha^V$. We have that $F_{\mu_\alpha}'(y) =f_\alpha^V(y)$, so that this implies that each rescaling change of variables is bi-Lipschitz.

Let $\bar y = \sum_{\alpha=1}^m\lambda_\alpha\bar y^\alpha$, so that any layerwise-Wasserstein barycenter of the $\mu_\alpha$ is supported on $\mathbb{R}^d \times [0,\bar y]$. Assume that each $g^\alpha_{i_\alpha} \in C^1([\underline y_{i_\alpha}^\alpha,\overline y_{i_\alpha}^\alpha])$ and let $C$ be an upper bound on each $| (g^\alpha_{i_\alpha})'|$.  
%
We define the total root length of a layerwise-Wasserstein barycenter $Bar^{LW}(\sum_{\alpha=1}^m\lambda_\alpha\delta_{\mu_\alpha})$ as the one dimensional  Hausdorff measure of  its support, namely, 
the set 
$$
\{(x,y): x \in {\rm spt}(\tilde Bar^{LW}_l(\sum_{\alpha=1}^m\lambda_\alpha\delta_{\mu_\alpha})), l=(\sum_{\alpha=1}^m\lambda_\alpha F_{\mu_\alpha}^{-1})^{-1}(y)\},
$$
 where, as before $\tilde Bar^{LW}_l(\sum_{\alpha=1}^m\lambda_\alpha\delta_{\mu_\alpha})$ is the horizontal slice of the layerwise-Wasserstein barycenter at level $l$ (that is, the Wasserstein barycenter of the $(\tilde \mu_{\alpha})_l$).


Letting $G^\lambda_{i_1....i_m}$ be one of the graphs in the ghost, we let $\chi^\lambda_{i_1....i_m}$ be the indicator function of its active set, that is, 
 
\begin{align*}
 \chi^\lambda_{i_1....i_m}(y) =
\begin{cases}
1     & \text{if $G^\lambda_{i_1....i_m}(y)$ is well defined}\\
& \text{ and in the support of $\tilde Bar^{LW}_{(\sum_{\alpha=1}^m\lambda_\alpha F_{\mu_\alpha}^{-1})^{-1}(y)}(\sum_{\alpha=1}^m\lambda_\alpha\delta_{\mu_\alpha})$}, \\
    0  & \text{otherwise}.
\end{cases}
\end{align*}

 The root length is then
$$
\sum_{i_1,...i_m} \int_0^{\bar y} \sqrt{1+| (G^\lambda_{i_1....i_m})'(y)|^2}\chi^\lambda_{i_1....i_m}(y) dy.
$$

\begin{proposition}\label{prop:root-length}
	Letting $\mu_\alpha$ be skeletal roots for $\alpha =1,2...m$, there is a layerwise -Wasserstein barycenter $
	Bar^{LW}(\sum_{\alpha=1}^m\lambda_\alpha\delta_{\mu_\alpha})$ of $\sum_{\alpha=1}^m\lambda_\alpha\delta_{\mu_\alpha}$ for which
	
	$$
	C_0 R(\mu_\beta) \leq R\left(Bar^{LW}\left(\sum_{\alpha=1}^m\lambda_\alpha\delta_{\mu_\alpha}\right)\right) \leq C_1\left[C_2\sum_{\alpha=1}^m R(\mu_\alpha)-(m-1)\right]
	$$
	for any $\beta =1,2....,m$. The constants $C_0,C_1$ and $C_2$ depend only on $C=\sup_{\alpha, i_\alpha}||(g_{i_\alpha}^\alpha)'||_{L^\infty}$, $L$ and $U$.
\end{proposition}

The proof essentially consists of two steps: first, we establish a similar result for the rescaled versions.  We then use bounds on the rescaling change of variables to translate the rescaled inequalities back to the original coordinates.  
We isolate the first step as a separate lemma.
\begin{lemma}\label{lem:rescaled-root-length}
	Using the notation in the Proposition above, there exists a layerwise Wasserstein barycenter such that
	$$
	\tilde C_0 R(\tilde \mu_\beta) \leq R\left(\tilde Bar^{LW}\left(\sum_{\alpha=1}^m\lambda_\alpha\delta_{\mu_\alpha}\right)\right) \leq \tilde C_1\left[\sum_{\alpha=1}^mR(\tilde \mu_\alpha)-(m-1)\right]
	$$
	for any $\beta =1,2....,m$.  
\end{lemma}
\begin{proof}
	Note that $\tilde \mu_\alpha$ is supported on the skeletal set $\cup_{i=1}^{N_{\alpha}}{\rm }(\tilde g^\alpha_{i_\alpha})$, where $\tilde g^\alpha_{i_\alpha}:=g^\alpha_{i_\alpha} \circ F_{\mu_\alpha}^{-1}$. The $\tilde g^\alpha_{i_\alpha}$ then have derivatives bounded by $\tilde C =C/L$.
	
	The ghost of the rescaled system consists of the limbs $\tilde G^\lambda_{i_1....i_m}=\sum_{\alpha=1}^{m}\lambda_\alpha \tilde g^\alpha_{i_1}$, which inherit the same derivative bounds as the $\tilde g^\alpha_{i_\alpha}$, $|(\tilde G^\lambda_{i_1....i_m})'| \leq \tilde C$, and at each $l \in[0,1]$ it is shown in \cite{AnderesBogwardtMiller16}  that there is a Wasserstein barycenter $Bar^W(\sum_{\alpha}\lambda_\alpha\delta_{(\tilde\mu_\alpha)_l})$ of the discrete measures $(\tilde \mu_\alpha)_l$ such that the number $S(l)$ of points in its support is at most 
	$\sum_{\alpha=1}^mS_\alpha(y) -m+1$ , where $S_\alpha(l)$ is the number of points in the support of  $(\tilde \mu_\alpha)_l$.  We use this Wasserstein barycenter in our construction of $\tilde Bar^{LW}(\sum_{\alpha=1}^m\lambda_\alpha\delta_{\mu_\alpha})$.  Therefore,
	\begin{eqnarray*}
	& & 	R\left(\tilde Bar^{LW}\left(\sum_{\alpha=1}^m\lambda_\alpha\delta_{\mu_\alpha}\right)\right)\\&=&\sum_{i_1,...i_m} \int_0^{1} \sqrt{1+| (\tilde G^\lambda_{i_1....i_m})'(l)|^2}\chi^\lambda_{i_1....i_m}(l) dl\\
		&\leq&\int_0^{1}\sqrt{1+\tilde C^2}\left[\sum_{\alpha=1}^mS_\alpha(l) -m+1\right]dl\\
		&\leq&\int_0^{1}\sum_{\alpha=1}^m\sum_{i_\alpha =1}^{N_\alpha} \sqrt{1+| (\tilde g^{\alpha}_{i_\alpha})'(l)}|^2\sqrt{1+\tilde C^2}\chi_{i_\alpha}^\alpha(l)dl-\sqrt{1+\tilde C^2}(m-1)\\
		&=&\sqrt{1+\tilde C^2}\sum_{\alpha=1}^mR(\tilde \mu_\alpha)-\sqrt{1+\tilde C^2}(m-1)
	\end{eqnarray*}	
	Similarly, the $S(l)$ is bounded below by the support of each marginal, $S(l) \geq S_\beta(l)$, and so, for each $\beta$
	\begin{eqnarray*}
		R\left(\tilde Bar^{LW}\left(\sum_{\alpha=1}^m\lambda_\alpha\delta_{\mu_\alpha}\right)\right) &=&\sum_{i_1,...i_m} \int_0^{1} \sqrt{1+| (\tilde G^\lambda_{i_1....i_m})'(l)|^2}\chi^\lambda_{i_1....i_m}(l) dl\\
		&\geq&\int_0^{1}S_\beta(l)dl\\
		&\geq&\int_0^{1}\sum_{i =1}^{N_\beta}\frac{ \sqrt{1+| (\tilde g^{\beta}_i)'(l)}|^2}{\sqrt{1+\tilde C^2}}\chi_{i}^\beta(l)dl\\
		&=&\frac{1}{\sqrt{1+\tilde C^2}}R(\tilde \mu_\beta).
	\end{eqnarray*}.
\end{proof}

The proof of the proposition combines the lemma with straightforward estimates on the change of variables $F_{\mu_\alpha}$.
\begin{proof}[Proof of Proposition~\ref{prop:root-length}]
	The root length of each limb  satisfies:
	\begin{eqnarray*}
		\int_{\underline y_i}^{\overline y_i}\sqrt{1+[(g_{i_\alpha}^\alpha)'(y)]^2} dy &=&\int_{\underline l_i}^{\overline l_i}\sqrt{1+[(g_{i_\alpha}^\alpha)'(F_{\mu_\alpha}^{-1}(l))]^2} [F_{\mu_\alpha}^{-1}]'(l)dl\\
		&=&\int_{\underline l_i}^{\overline l_i}\sqrt{[F_{\mu_\alpha}^{-1}]'(l)^2+[(g_{i_\alpha}^\alpha)'(F_{\mu_\alpha}^{-1}(l))]^2[F_{\mu_\alpha}^{-1}]'(l)^2} dl\\
		&\leq&K\int_{\underline l_i}^{\overline l_i}\sqrt{[1+[(g_{i_\alpha}^\alpha)'(F_{\mu_\alpha}^{-1}(l))]^2[F_{\mu_\alpha}^{-1}]'(l)^2} dl\\
	\end{eqnarray*}
	where $K=\max(\frac{1}{L},1)$.  The last term corresponds to the root length of the corresponding limb of $\tilde \mu_\alpha$.  Adding over all limbs we get
	$$
	R(\mu_\alpha) \leq KR(\tilde \mu_\alpha),
	$$
	while a symmetric argument yields 
	
	$$
	R(\mu_\alpha) \geq kR(\tilde \mu_\alpha),
	$$
	with $k=\min (1/U,1)$. 
	
	Similarly, since the vertical rescaling for the barycenter $$F_{Bar^{LW}\left(\sum_{\alpha}\lambda_\alpha \delta_{\mu_\alpha}\right)}=(\sum_{\alpha =1}^m\lambda_\alpha F_{\mu_\alpha}^{-1})^{-1}$$ inherits first derivative bounds from the $\mu_i$, we also get 
	$$
	kR\left(\tilde Bar^{LW}(\sum_{\alpha=1}^m\lambda_\alpha\delta_{\mu_\alpha})\right)\leq R\left(Bar^{LW}(\sum_{\alpha=1}^m\lambda_\alpha\delta_{\mu_\alpha})\right) \leq KR\left(\tilde Bar^{LW}(\sum_{\alpha=1}^m\lambda_\alpha\delta_{\mu_\alpha})\right).
	$$
	
	Combined with the Lemma~\ref{lem:rescaled-root-length}, these estimates yield the desired result.
\end{proof}
\begin{remark}
The result also holds, with essentially the same proof, for weak skeletal root measures, provided we take  $\chi_i^\alpha$ in \eqref{eqn: root length} to be the indicator function of the subset of the domain $[\underline y_i^\alpha, \overline y_i^\alpha]$ where $(\tilde \mu_\alpha)_l(g_{i_\alpha}^\alpha(y)) >0$, for $l=F_\mu(y)$.  
\end{remark}

\begin{remark} 	
 It is unfortunately not possible to establish an upper bound on the root length of the layerwise Wasserstein barycentre which is \emph{independent} of the number of samples $m$.
 
 To see this, consider the skeletal root measures $\mu_\alpha$, each concentrated on two curves $g_{\alpha_1}, g_{\alpha_2}: [0,1]\rightarrow \mathbb{R}$, with $g_{\alpha_1}(y)=0$ and $g_{\alpha_2}(y)=y$..  We let the one dimensional density of each $\mu_i$ be constant on each of the two limbs, with densities $\frac{1}{\alpha}$ on $g_{\alpha_1}$ and $1-\frac{1}{\alpha}$ on $g_{\alpha_2}$ (normalized to have total mass $1$).  The vertical marginals of each $\mu_\alpha$ are then uniform, so $ F_{\mu_\alpha}(y)=y$ and each $(\tilde \mu_{\alpha})_l=\frac{1}{\alpha}\delta_0 +(1-\frac{1}{\alpha})\delta_l$.
 
 It is then not hard to see that the Wasserstein barycenter of $\sum_{\alpha=1}^{m}\frac{1}{m}(\tilde \mu_{\alpha})_l$ is then concentrated on the $m+1$ points $0,\frac{l}{m}, \frac{2l}{m},...l$, and so the support of the layerwise-Wasserstein barycenter of $\sum_{\alpha=1}^{m}\frac{1}{m} \mu_{\alpha}$  consists of the $m$ curves $g_1,...,g_m: [0,1]\rightarrow \mathbb{R}$, with $g_\alpha(y) =\frac{\alpha y}{m}$; the total root length clearly grows with $m$.
 
Interpolating between a large number of marginals, or samples, $m$, can therefore result in weak skeletal measures with very large total root length,

\end{remark}

\section{Layerwise Wasserstein convexity}\label{sec:displacementconvexity}
 We will call a function $\mathcal F: P(\mathbb{R}^d \times \R_{\ge 0})\rightarrow \mathbb{R}$  \emph{layerwise- Wasserstein 
	convex} if for any $\Omega \in P(P(\R^d \times \R_{\ge 0}))$, 
	\begin{align*}
	\mathcal{F} \left( Bar^{LW}(\Omega) \right) \le \int \mathcal{F}(\mu) d\Omega(\mu).  
	\end{align*} This notion of convexity 
	 may potentially play an important role in applications.  Given a family of root systems, corresponding to a family of  genetically identical plants, grown under identical environmental conditions, we will in forthcoming work propose interpreting the layerwise-Wasserstein barycenter of the systems as the single root system which best represents the family.  It is natural to compare phenotypes (for instance, center of mass, variance, entropy, total root length, etc.) of that barycenter with the phenotypes of the actual observed roots in the original family.  If these phenotypes (interpreted as functionals on the space of measures) are layerwise-Wasserstein convex, 
	 the phenotype of the barycenter is always less than the average of the phenotypes of the samples. 

The theory of layerwise convexity, which we begin to develop below, has a strong connection to the theory of displacement convexity, or convexity along geodesics with respect to the Wasserstein metric, introduced by McCann \cite{McCann97}, and its extension to convexity over Wasserstein barycenters, introduced by Agueh-Carlier \cite{AguehCarlier11}.  

We begin with the Shannon entropy, perhaps the best known  displacement convex functional.  As we show below, it is also layerwise-Wasserstein convex.

For roots, it can be regarded as a measure of  the concentration of mass and therefore has potential biological interest. 
	Given $\mu \in P(\R^{d}\times\R_{\ge 0})$, with $$\mu(x,y) = f(x,y)dxdy, \hbox{ where $x \in \mathbb{R}^n$, $y \in \mathbb{R}_{\ge 0}$,}$$  the vertical marginal $\mu^V$ has density $$f^V(y) = \int_{\mathbb{R}}f(x,y)dx.$$  Note that for fixed $y$, the probability measure $$d\mu_y(x) = \frac{f(x,y)}{f^V(y)}dx$$ coincides with $\tilde \mu_l$ for $l=F_\mu(y)$.
Recall that the Shannon entropy of $\mu$ is defined as $$S(\mu) = \int_{\mathbb{R}^n} \int_{\mathbb{R}_{\geq 0}}f(x,y)\log f(x,y)dxdy.$$
This formula allows one to rewrite $S(\mu)$ using the layerwise decomposition.
\begin{proposition}
	The Shannon entropy $S(\mu)$ satisfies:
	\begin{align*}
	S(\mu)&= \int_{\mathbb{R}_{\geq 0}}S(\mu_y)d\mu^V(y) +S(\mu^V)\\
	&=\int_{0}^1S(\tilde \mu_{l})dl +S(\mu^V).
	\end{align*}	
\end{proposition}
\begin{proof}
	The proof is a calculation. 
	In the following we use the  standard convention that $0\log(0) =0$.
		We have
	\begin{eqnarray*}
		S(\mu) &=& \int_{\mathbb{R}_{\geq 0}} \int_{\mathbb{R}^n}f(x,y)\log[ f(x,y)f^V(y)/f^V(y)]dxdy\\
		&=&\int_{\mathbb{R}_{\geq 0}} \int_{\mathbb{R}^n}f(x,y)\left[\log\left[ \frac{f(x,y)}{f^V(y)}\right]+\log f^V(y)\right]dxdy\\
		&=&\int_{\mathbb{R}_{\geq 0}} \int_{\mathbb{R}^n}f(x,y)\log\left[ \frac{f(x,y)}{f^V(y)}\right]dxdy+ \int_{\mathbb{R}}\Big(\int_{\mathbb{R}^n}f(x,y)dx \Big)\log f^V(y)dy\\
		&=&\int_{\mathbb{R}_{\geq 0}} \int_{\mathbb{R}^n}f^V(y)\frac{f(x,y)}{f^V(y)}\log\left[ \frac{f(x,y)}{f^V(y)}\right]dxdy+ \int_{\mathbb{R}}f^V(y)\log f^V(y)dxdy\\
		&=&\int_{\mathbb{R}_{\geq 0}} \int_{\mathbb{R}^n}\Big(\frac{f(x,y)}{f^V(y)}\log\left[ \frac{f(x,y)}{f^V(y)}\right]dx\Big)f^V(y)dy+ S(\mu^V)\\
		&=&\int_{\mathbb{R}_{\geq 0}} S(\mu_y)d\mu^V(y)+ S(\mu^V).
	\end{eqnarray*}
	The final equality follows by noting that the cumulative distribution function $F_{\mu}$ satisfies $F_{\mu}'(y) =f^V(y)$ and changing variables from $y$ to $l=F_{\mu}(y)$.
\end{proof}

\begin{corollary}
	The Shannon entropy is layerwise-Wasserstein convex.
\end{corollary}
\begin{proof}
	Recall that  from Proposition~\ref{prop:decomp} the layerwise-Wasserstein interpolation of $\Omega \in P(P(\R^d \times \R_{\ge 0}))$
  amounts to constructing the  probability measure $\eta$ whose vertical marginal $\eta^V$
	 is the Wasserstein barycenter of $\Omega^V$, and whose conditional probabilities $\tilde \eta_l$ 
	 are the Wasserstein barycenters of the $\tilde \Omega_l$. 
	  Since the entropy $S(\mu)$ 
	depends additively on  $\mu^V$ and the $\tilde \mu_l$,  Wasserstein 
	convexity
	 of the entropy (see \cite{KimPass17} for convexity with respect to general barycenters) yields the result.
\end{proof}

Many phenotypes of interest concern only the depth of the root, and not its horizontal distribution of mass (since, for instance, nutrient concentration in soil is largely determined by depth).  Therefore the following simple observation is relevant. 
\begin{proposition}
	Any Wasserstein convex function of the vertical marginal is layerwise  Wasserstein convex.
\end{proposition}
\begin{proof}
 This follows immediately from the structure of the layerwise-Wasserstein distance,  given in Proposition~\ref{prop:decomp}. 
\end{proof}
Let us list a few examples of functionals with possible biological applications, which are layerwise-Wasserstein convex by this proposition:
\begin{example}
\begin{itemize}
 \item 
 The vertical mean, $\mu \mapsto \bar y:=\int_{\mathbb{R}^d \times \mathbb{R}_{\geq 0 }}yd\mu(x,y) =\int_{\mathbb{R}_{\geq 0 }}yd\mu^V(y)$;  one can verify easily that this is in fact affine along displacement (and hence layerwise-Wasserstein) interpolations.

 \item  The vertical variance $\int_{\mathbb{R}^d\times\mathbb R_{\geq 0 }}|y-\bar y|^2 d\mu(x,y) =\int_{\mathbb R_{\geq 0 }}|y-\bar y|^2 d\mu^V(y)$, a measure the spread of the mass in the vertical direction \cite{KimPass15}. 
\item  The vertical internal energy $$\int_{\mathbb{R}_{\geq0}} (f^V(y))^r dy \quad \hbox{ for } r\geq 1.$$  

\item    Vertical quantiles $F^{-1}_\mu(l)$ for each fixed $l \in (0,1)$. For instance, the vertical median ($l=1/2$) is the depth above which half the mass of the root lies.  The 100th quantile (the maximal depth of the root) is often called the rooting depth, while the 87th quantile ($l =87/100$) is a conventional phenotype often used as a measure of the root depth.  The displacement convexity (and layerwise-Wasserstein  convexity) of these follows immediately from the monotone structure of one-dimensional optimal couplings  with respect to the distance squared cost; in fact, it is layerwise-Wasserstein  affine.
\end{itemize}

\end{example}
Although, unlike the examples above, it is not a functional of the vertical marginal, the structure of the layerwise-Wasserstein distance easily implies that the class of functionals in the following example below are layerwise-Wasserstein convex as well.
\begin{example}
	Any functional of the form
$$
\bar{\mathcal{F}}(\mu) = \int_0^1 \bar{\mathcal{F}}_l(\tilde \mu_l )dl.
$$
where each $\bar{\mathcal{F}}_l$ is  Wasserstein convex is layerwise-Wasserstein convex.
\end{example}

\bibliographystyle{plain}

\bibliography{bibliography}

\end{document}